\documentclass[12pt,reqno]{amsart}

\usepackage[dvips]{graphicx}
\usepackage[arrow,matrix,curve]{xy} 

\usepackage{amssymb, latexsym, amsmath, amscd, array, 
}

\newcommand\Z {{\mathbb Z}} 
\newcommand\R {{\mathbb R}} 
\newcommand\sys{{\rm Sys}} 
\newcommand\area{{\rm Area}} 
\newcommand\length{{\rm length}}

\newtheorem{theorem}{Theorem}[section]
\newtheorem{lemma}[theorem]{Lemma}
\newtheorem{proposition}[theorem]{Proposition}
\newtheorem{corollary}[theorem]{Corollary}

\theoremstyle{definition}

\newtheorem{remark}[theorem]{Remark}
\newtheorem{question}[theorem]{Question}

\numberwithin{equation}{section}
\numberwithin{figure}{section} 
\numberwithin{table}{section}

\newcommand\var {{\rm Var}}
\newcommand\gmetric{{g}}

\newcommand\C{{\mathbb {C}}}

\DeclareMathOperator{\LB}{{LB}}
\DeclareMathOperator{\trace}{{\rm trace}}
\newcommand{\Hess}{{\rm Hess}}
\DeclareMathOperator{\TTT}{{T}}
\newcommand\T {{\mathbb T}} 

\DeclareMathOperator{\av}{{av}}
\newcommand{\fla}{{f_{\LA}}} 
\DeclareMathOperator{\LA}{{LA}}

\begin{document}

\title{Liouville's equation for curvature and systolic defect}

\author{Mikhail Katz}

\address{Department of Mathematics, Bar Ilan University, Israel}
\email{katzmik@macs.biu.ac.il}

\footnotetext{Supported by the Israel Science Foundation grant
1294/06}



\bigskip\bigskip\bigskip\noindent

\begin{abstract}
We analyze the probabilistic variance of a solution of Liouville's
equation for curvature, given suitable bounds on the Gaussian
curvature.  The related systolic geometry was recently studied by
Horowitz, Katz, and Katz in~\cite{HKK}, where we obtained a
strengthening of Loewner's torus inequality containing a ``defect
term", similar to Bonnesen's strengthening of the isoperimetric
inequality.  Here the analogous isosystolic defect term depends on the
metric and ``measures" its deviation from being flat.  Namely, the
defect is the variance of the function $f$ which appears as the
conformal factor expressing the metric on the torus as
$f^2(x,y)(dx^2+dy^2)$, in terms of the flat unit-area metric in its
conformal class.  A key tool turns out to be the computational formula
for probabilistic variance, which is a kind of a sharpened version of
the Cauchy-Schwartz inequality.
\end{abstract}

\maketitle 

\tableofcontents

\section{Liouville's equation for Gaussian curvature}

Given a function~$f$ satisfying Liouville's equation for curvature in
a domain, we are interested in studying lower bounds for the
probabilistic variance of~$f$, or more precisely, for the average
square deviation of~$f$ from its mean, in the domain.

A geometric application we have in mind is obtaining lower bounds for
the variance of the conformal factor~$f$ in a fundamental domain of a
doubly-periodic metric~$f^2(x,y)(dx^2+dy^2)$ in the plane, and hence
for the isosystolic defect in Loewner's inequality for the
corresponding torus (see Section~\ref{two}).  

Liouville's equation for curvature is usually stated in terms of the
Laplace-Beltrami operator~$\Delta_{\LB}$ on a surface with a
metric~$f^2(x,y) ds^2$, where~$ds^2=dx^2+dy^2$.  In isothermal
coordinates~$(x,y)$, the operator is given
by~$\Delta_{\LB}^{\phantom{I}}=
\frac{1}{f^2}\left(\frac{\partial^2}{\partial
x^2}+\frac{\partial^2}{\partial y^2} \right)$.  We obtain the
following form of Liouville's equation, see \cite{DFN},
\cite[p.~26]{SGT} for details.

\begin{theorem}[Liouville's equation]
The Gaussian curvature~$K=K(x,y)$ of the metric~$f^2(x,y)(dx^2+dy^2)$
is minus the Laplace-Beltrami operator of the~$\log$ of the conformal
factor~$f$:
\begin{equation}
\label{232} 
K=-\Delta_{\LB}^{\phantom{I}}\log f.
\end{equation}
\end{theorem}
In terms of the {\em flat\/} Laplacian~$\Delta_0(h)=\trace\Hess(h)$,
the equation can be written as follows:~$-\Delta_0\log f=Kf^2$.
Setting~$h=\log f$, we obtain yet another equivalent
form,~$-\Delta_0h=Ke^{2h}$.

This equation in the case of constant curvature~$K$ is called
Liouville's equation in \cite[p.~118]{DFN}, see also Rogers and
Schief~\cite[p.~154]{RS} (where it appears as the first
Gauss-Mainardi-Codazzi equation), and \cite{LS,
Po,Ve}.

Liouville's equation~$K=-\Delta_{\LB} \log f$ can be written as
follows in terms of partial
derivatives:
\begin{equation}
\label{12}
-f(f_{xx}+f_{yy})+(f_x^2+f_y^2)=Kf^4,
\end{equation}
in other words~$-f\Delta_0 f+|\nabla f|^2=Kf^4$.  The flat Laplacian
can be written as follows in polar coordinates:
\begin{equation}
\label{32}
\Delta_0 f = \frac{\partial^2 f}{\partial
r^2}+\frac{1}{r}\frac{\partial f}{\partial
r}+\frac{1}{r^2}\Delta_{S^1}f,
\end{equation}
where~$\Delta_{S^1}f$ is the Laplace-Beltrami operator on the circle.

\section{A geometric application}
\label{two}

In this section, we present a differential-geometric application of
lower bounds for the variance of a solution of Liouville's equation
for curvature.  The reader mostly interested in Liouville's equation
itself can skip to the next section.  

The {\em systole\/} of a compact metric space~$X$ is a metric
invariant of~$X$, defined to be the least length of a noncontractible
loop in~$X$.  We will denote it~$\sys=\sys(X)$, cf.
M.~Gromov~\cite{Gr1, Gr3}.  When~$X$ is a graph, the invariant is
usually referred to as the {\em girth\/}, ever since W.~Tutte's
article~\cite{Tu}.  C.~Loewner proved his systolic inequality
%
%
$\area(\gmetric) - \tfrac{\sqrt{3}}{2} \sys(\gmetric)^2 \geq 0$ for
the torus~$(T^2,g)$ in 1949, as reported by Pu \cite{Pu}.

The classical Bonnesen inequality from 1921 is the strengthened
isoperimetric inequality~$ L^2 - 4\pi A \geq \pi^2(R-r)^2$, see
\cite[p.~3]{B-Z88}, \cite{Bo}.  Here~$A$ is the area of the region
bounded by a closed Jordan curve of length (perimeter)~$L$ in the
plane,~$R$ is the circumradius of the bounded region, and~$r$ is its
inradius.  The error term~$\pi^2(R-r)^2$ on the right hand side of
Bonnesen's inequality is traditionally called the {\em isoperimetric
defect}.

Loewner's torus inequality can be similarly strengthened, by
introducing a ``defect" term \`a la Bonnesen.  If we use conformal
representation to express the metric~$\gmetric$ on the torus~$T^2$
as~$f^2 (dx^2+dy^2)$ with respect to a unit area flat metric on the
torus $\R^2/L$ (see below), then the defect term in question is the
variance of the conformal factor~$f$ above.  Then the inequality with
the defect term looks as follows:
\[
\area-\tfrac{\sqrt{3}}{2}\sys^2\geq\var(f).
\]
Is there a geometrically meaningful estimate for the systolic defect
term in Loewner's torus inequality?  Thus, one could look for lower
bounds for the variance of the conformal factor~$f$ of the metric, in
terms of some curvature conditions, say if the curvature is bounded
away from zero on a region~$D$ whose area is bounded below.
Liouville's equation for Gaussian curvature~$K$ is~$-\Delta\log f = K
f^2$.  One is led to the following problem in connection with
Liouville's equation.  The solutions in the constant curvature case
are of the form~$f=\frac{|a'(z)|}{1+|a(z)|^2}$ where~$a(z)$ is a
holomorphic function on a disk (here the curvature, assumed constant,
is normalized to the value +4).  One seeks lower bounds for the
variance of~$f$.  Here a lower bound for the area of the region~$D$
translates into a lower bound for the~$L^2$ norm
of~$\frac{|a'(z)|}{1+|a(z)|^2}$.


Recall that the uniformisation theorem in the genus~$1$ case can be
formulated as follows: For every metric~$\gmetric$ on
the~$2$-torus~$T^2$, there exists a lattice~$L\subset \R^2$ and a
positive~$L$-periodic function~$f(x,y)$ on~$\R^2$ such that the
torus~$(T^2,\gmetric)$ is isometric to~$\left( \R^2/L, f^2
ds^2\right)$, where~$ds^2=dx^2+dy^2$ is the standard flat metric
of~$\R^2$.  Similarly to the isoperimetric inequality, Loewner's torus
inequality relates the total area, to a suitable~$1$-dimensional
invariant, namely the systole, i.e., least length of a noncontractible
loop on the torus~$(T^2, \gmetric)$:
\begin{equation}
\label{11L} 
\area(\gmetric) - \tfrac{\sqrt{3}}{2} \sys(\gmetric)^2 \geq 0.
\end{equation}
In analogy with Bonnesen's inequality, there exists a
following version of Loewner's torus inequality with an error term:
\begin{equation}
\label{13}
\area(\gmetric) - \tfrac{\sqrt{3}}{2} \sys(\gmetric)^2 \geq \var(f),
\end{equation}
see \cite{HKK}.  Here the error term, or {\em isosystolic defect}, is
given by the variance~$\var(f)=\int_{T^2} (f-m)^2$ of the conformal
factor~$f$ of the metric~$\gmetric=f^2\gmetric_0$ on the torus,
relative to the unit area flat metric~$\gmetric_0$ in the same
conformal class.  Here~$m=\int_{T^2}f$ is the mean of~$f$.  More
concretely, if~$(T^2,\gmetric_0) = \R^2/L$, where~$L$ is a lattice of
unit coarea, and~$D$ is a fundamental domain for the action of~$L$
on~$\R^2$ by translations, then the mean can be written as~$m=\int_D
f(x,y) dxdy$, where~$dxdy$ is the standard measure of~$\R^2$.
\begin{question}
Unlike Bonnesen's inequality where the error term has clear geometric
significance, the error term in \eqref{13} is of an analytic nature.
It would be interesting to obtain a lower bound whose geometric
significance is more transparent.  Can such a bound be expressed in
terms of suitable curvature bounds?
\end{question}

The proof of inequalities with isosystolic defect relies upon the
computational formula for the variance of a random variable in terms
of expected values.  Keeping our differential geometric application in
mind, we will denote the random variable~$f$.  Namely, we have the
formula
\begin{equation}
\label{11}
E_\mu(f^2) - \left(E_\mu(f) \right)^2 = \var(f),
\end{equation}
where~$\mu$ is a probability measure.  Here the variance is~$ \var(f)=
E_\mu \left( (f-m)^2 \right)$, where~$m=E_\mu(f)$ is the expected
value (i.e., the mean).

Now consider a flat metric~$\gmetric_0$ of unit area on the~$2$-torus
$T^2$.  Denote the associated measure by~$\mu$.  Since~$\mu$ is a
probability measure, we can apply formula~\eqref{11} to it.  Consider
a metric~$\gmetric = f^2 \gmetric_0$ conformal to the flat one, with
conformal factor~$f>0$, and new measure~$f^2\mu$.  Then we have~$
E_\mu (f^2) = \int_{T^2} f^2 \mu = \area (\gmetric). ~$
Equation~\eqref{11} therefore becomes
\begin{equation}
\label{21}
\area(\gmetric) - \left( E_\mu(f) \right)^2 = \var(f).
\end{equation}
Next, one relates the expected value~$E_\mu(f)$ to the systole of the
metric~$\gmetric$.  Recall that the first successive
minimum,~$\lambda_1(L)$ is the least length of a nonzero vector
in~$L$.  The lattice of the {\em Eisenstein integers\/} is the lattice
in~$\C$ spanned by the elements~$1$ and the sixth root of unity.  To
visualize this lattice, start with an equilateral triangle in~$\C$
with vertices~$0$,~$1$, and~$\tfrac{1}{2}+i\tfrac{\sqrt{3}}{2}$, and
construct a tiling of the plane by repeatedly reflecting in all sides.
The Eisenstein integers are by definition the set of vertices of the
resulting tiling.  The maximal ratio~$\frac{(\lambda_1)^2}{\area}$
(the area is that of a fundamental domain) for a lattice in the plane
is~$\gamma_2=\frac{2}{\sqrt{3}}=1.1547\ldots$.  The corresponding
critical lattice is homothetic to the~$\Z$-span of the cube roots of
unity in~$\C$, i.e., the Eisenstein integers.  This result yields a
proof of Loewner's torus inequality for the metric~$\gmetric = f^2
\gmetric_0$ using the computational formula for the variance.  Let us
analyze the expected value term~$E_\mu(f) = \int_{T^2} f \mu$ in
\eqref{21}.  Indeed, the lattice of deck transformations of the flat
torus~$\gmetric_0$ admits a~$\Z$-basis similar to~$\{\tau, 1\} \subset
\C$, where~$\tau$ belongs to the standard fundamental domain.  In
other words, the lattice is similar to~$ \Z\tau + \Z1 \subset \C. ~$
Consider the imaginary part~$\Im(\tau)$ and set~$
\sigma^2:=\Im(\tau)>0. ~$ From the geometry of the fundamental domain
it follows that~$\sigma^2 \geq \tfrac{\sqrt{3}}{2}$, with equality if
and only if~$\tau$ is the primitive cube or sixth root of unity.
Since~$\gmetric_0$ is assumed to be of unit area, the basis for its
group of deck tranformations can therefore be taken to
be~$\{\sigma^{-1}\tau,\sigma^{-1}\}$,
where~$\Im({\sigma}^{-1}{\tau})=\sigma$.  With these normalisations,
we see that the flat torus is ruled by a pencil of horizontal closed
geodesics, denoted~$\gamma_y=\gamma_y(x)$, each of
length~$\sigma^{-1}$, where the ``width'' of the pencil
equals~$\sigma$, i.e. the parameter~$y$ ranges through the
interval~$[0,\sigma]$, with~$\gamma_\sigma=\gamma_0$.  By Fubini's
theorem, we obtain the following lower bound for the expected
value:~$E_\mu(f) = \int_0^\sigma \left( \int_{\gamma_y} f(x)dx \right)
dy = \int_0^\sigma \length(\gamma_y)dy \geq \sigma \sys(\gmetric),~$
see \cite[p.~41, 44]{SGT} for details.  Substituting into \eqref{21},
we obtain the inequality
\begin{equation}
\label{51b}
\area(\gmetric) - \sigma^2 \sys(\gmetric)^2 \geq \var(f),
\end{equation}
where~$f$ is the conformal factor of the metric~$\gmetric$ with
respect to the unit area flat metric~$\gmetric_0$.  Since~$\sigma^2
\geq \tfrac{\sqrt{3}}{2}$, we obtain in particular Loewner's torus
inequality with isosystolic defect~\eqref{13}.  Hence, a metric
satisfying the boundary case of equality in Loewner's torus
inequality~\eqref{11L} is necessarily flat and homothetic to the
quotient of~$\R^2$ by the lattice of Eisenstein integers.  Indeed, if
a metric~$f^2 ds^2$ satisfies the boundary case of equality in
\eqref{11L}, then the variance of the conformal factor~$f$ must vanish
by \eqref{13}.  Hence~$f$ is a constant function.

If~$\tau$ is pure imaginary, i.e., the lattice~$L$ is a rectangular
lattice of coarea~$1$, then the metric~$\gmetric=f^2 \gmetric_0$
satisfies the inequality
\begin{equation}
\label{52z}
\area(\gmetric) - \sys(\gmetric)^2 \geq \var(f).
\end{equation}
Indeed, if~$\tau$ is pure imaginary then~$\sigma\geq 1$, and the
inequality follows from \eqref{51b}.  In particular, every surface of
revolution satisfies \eqref{52z}, since its lattice is rectangular.

Lower bounds for the variance of the conformal factor yield lower
bounds for the isosystolic defect, see next section.

\section{Rotationally invariant case}

If~$f$ is rotationally invariant, then~$\Delta_0 f = \frac{\partial^2
f}{\partial r^2} + \frac{1}{r} \frac{\partial f}{\partial r}$ from
\eqref{32}.  Thinking of~$f$ as a single-variable
function~$f(r,\theta)=f(r)$, we obtain
\begin{equation}
\label{86}
\Delta_0 f = f''(r) + \frac{1}{r} f'(r).
\end{equation}
Liouville's equation then becomes an ordinary differential
equation
\[
-f\left(f''(r)+\frac{1}{r}f'(r) \right) + (f'(r))^2 = K f^4(r),
\]
cf.~\eqref{12}.  The substitution~$\zeta=r^2$ results in an equation
in the variable~$\zeta$, namely
\[
-f\TTT(f)+4\zeta(f'(\zeta))^2=Kf^4,
\]
where the relation between $f(r)$ and $f(\zeta)$ needs to be
explained (two different f's).

Alternatively, we can proceed as follows.  Instead of the substitution
above, we study the variance of the product~$rf(r)$, where~$f$ is the
conformal factor.  Thus we retain the variable~$r$.  Is there a
convenient form of the equation for this new function~$rf(r)$?
However, the natural~$L^2$ normalisation is for~$\int f^2(r) r dr
d\theta$, hence here we need to integrate~$rf^2(r)$.  In other words
we are integrating~$rf(r)$ with a weight function~$1/r$.  We consider
again the operator~\eqref{86}, which we now
denote~$\TTT$:~$\TTT=\frac{d^2}{dr^2}+\frac{1}{r}\frac{d}{dr}$,
where~$r\geq 0$.  Applying the change of variable~$\zeta=r^2$, we can
write~$\TTT$ as~$\TTT= 4 \frac{d}{d\zeta} \zeta \frac{d}{d\zeta}$.
Furthermore the substitution~$\zeta=e^t$ allows us to write the
operator as
\begin{equation}
\label{102b}
\TTT= \frac{4}{\zeta} \frac{d^2}{dt^2}.
\end{equation}
Lower bounds for the second derivative lead easily to estimates for
the variance, expressed by the following Lemma: If the curvature is
positive, then the function~$f(r)$ is decreasing, while~$\log f$ is
concave with respect to the variable~$t$.  Indeed, rewriting the
equation~$- \T \log f = K f^2$ as~$ - \frac{d^2 \log f}{dt^2} =
\frac{\zeta}{4}K f^2~$ immediately implies the concavity of~$\log f$
with respect to~$t$.  Integrating with respect to~$t$, we obtain~$ -
\frac{d \log f}{dt} = \int \frac{e^t}{4}K f^2 >0~$ if the curvature is
positive.  Hence~$\log f$ is decreasing, and therefore so is~$f$
itself.  The significance of the variable~$\zeta$ stems from the
following elementary fact.
\begin{lemma}
\label{sss}
The variance of a rotationally invariant conformal factor~$f$ on a
disk is proportional to the variance of the corresponding single
variable function, with respect to the variable~$\zeta=r^2$.
\end{lemma}
Indeed, The proof is immediate from the fact that~$\frac{1}{2}
d\zeta=rdr$ is the line measure inherited from polar coordinates.
Liouville's equation for a rotationally invariant conformal
factor~$f=f(\zeta(r))$ can be rewritten as~$ 4 \frac{d}{d\zeta} \zeta
\frac{d}{d\zeta} \log f = -K f^2. ~$ In terms of the reciprocal
function~$\phi=\frac{1}{f}$, this becomes~$ 4 \phi^2 \frac{d}{d\zeta}
\zeta \frac{d}{d\zeta} \log \phi = K . ~$ Note that the {\em linear\/}
function~$\phi(\zeta)=1+\frac{K}{4} \zeta$ solves this equation, as
pointed out by B. Riemann, cf.  \eqref{14}.

\begin{question}
Can one translate the differential inequality
\[
4 \phi^2 \frac{d}{d\zeta} \zeta \frac{d}{d\zeta} \log \phi \geq K
\]
into a geometric condition involving a comparison of an arbitrary
solution, with Riemann's linear solution?
\end{question}

\section{Averaging~$h=\log f$ respects the differential inequality}

If the curvature satisfies a lower bound~$K(x,y)\geq \alpha$,
Liouville's equation yields a differential inequality
\begin{equation}
\label{83}
- \Delta_0 h\geq \alpha e^{2h(r)}.
\end{equation}

Given a disk in~$\R^2$ where~$h$ is defined, we can average~$h$ by the
circle of rotations about the center of the disk to obtain a
rotationally symmetric function~$h_{\av}(r)=h_{\av}(r, \theta)$ on the same
disk, where~$\theta$ is the polar angle.

\begin{proposition}
\label{83l}
If~$h$ satisfies the differential inequality~\eqref{83} then its
rotationally symmetric average~$h_{\av}(r)$ satisfies the ordinary
differential inequality
\begin{equation}
-\left (h_{\av}''(r) + \tfrac{1}{r} h_{\av}'(r) \right) \geq \alpha
e^{h_{\av}},
\end{equation}
where~$\alpha$ is a lower bound for Gaussian curvature.
\end{proposition}

\begin{proof}
By linearity of the flat Laplacian, we obtain from \eqref{83}~$
-\Delta_0(h_{\av}) = -\av \left( \Delta_0 h \right) = \av \left( K e^{2h}
\right) \geq \alpha \av \left(e^{2h} \right) \geq \alpha e^{2h_{\av}},
$ by Jensen's inequality applied to the exponential function.
Applying~\eqref{86}, we complete the proof of Proposition~\ref{83l}.
\end{proof}
The logarithmic average~$\fla$ is by definition the function~$ \fla=
e^{\av(\log f)}$.  The above proposition can be restated in terms of
the logarithmic average as follows.
\begin{corollary}
If the metric~$f^2ds^2$ admits a lower bound~$K \geq \alpha$ on a
disk, then the logarithmic average~$\fla$ satisfies the differential
inequality~$ -\Delta_{\LB} \log \fla \geq \alpha . ~$
\end{corollary}
Namely, at the level of the function~$h=\log f$, the differential
inequality~$ -\Delta_0 h \geq \alpha e^{2h}~$ averages well by
Jensen's inequality (see above).
\begin{remark}
One can find lower bounds for the variance of the conformal
factor~$f$, by relating the variances~$\var(f)$ and~$\var(\fla)$, and
developing variance estimates for a rotationally invariant function,
and applying them to the logarithmic average~$\fla$.
\end{remark}

\section{The effect of averaging on variance}

\begin{lemma}
Rotational averaging does not increase the variance.
\end{lemma}
Indeed, we claim that the rotational average~$h_{\av}$ of a function~$h$
on a disk~$D$ satisfies~$\var(h_{\av})\leq \var(h)$.  More specifically,
let~$ h_{\av}(r)=h_{\av}(r,\theta)= \frac{1}{2\pi} \int_0^{2\pi}
h(r,\theta)d\theta . ~$ Note that~$\int_0^{2\pi} h_{\av}(r) d\theta =
2\pi \frac{1}{2\pi} \int_0^{2\pi} h(r,\theta)d\theta = \int_0^{2\pi}
f(r,\theta)d\theta$.  Assume for simplicity that~$D$ has unit area.
Then
\begin{equation}
E(h_{\av}) 
=\int
rdr\left(\int_0^{2\pi}h_{\av}(r,\theta)d\theta\right)
=\int
rdr\left(\int_0^{2\pi}h(r,\theta)d\theta\right)
=E(h) .
\label{89}
\end{equation}
Now consider the term~$E(h^2)$.  We have~$ \int_0^{2\pi} h_{\av}^2(r)
d\theta = 2\pi h_{\av}^2(r) = 2\pi \left( \frac{1}{\pi} \int_0^{2\pi}
h(r,\theta)d\theta \right) ^2 . ~$ Since the squaring function is a
convex function, we can apply Jensen's inequality to obtain~$ \left(
\frac{1}{2\pi} \int h(r,\theta)d\theta \right) ^2 \leq \frac{1}{2\pi}
\int h^2(r,\theta)d\theta . ~$ Thus~$ \int_0^{2\pi} h_{\av}^2(r)d\theta
\leq 2\pi \frac{1}{2\pi} \int h^2(r,\theta)d\theta . ~$ Hence
\begin{equation}
E(h_{\av}^2) =\int\left(rdr\int_0^{2\pi}h_{\av}^2(r)d\theta\right) \leq
\int\left(rdr\int_0^{2\pi}h^2(r)d\theta \right) = E(f^2).
\label{810}
\end{equation}
Combining \eqref{89} and \eqref{810}, we obtain~$\var(h_{\av})=
E(h_{\av}^2)-E(h_{\av})^2=E(h_{\av}^2)-E(h)^2\leq E(h^2) - E(h)^2$, as
required.

\section{Some curvature estimates and Liouville's equation}

Assuming the existence of a region of positive Gaussian curvature
$K(x,y) \geq \alpha$ for the metric~$\gmetric$, we would like to
compare the variance of an arbitrary solution of Liouville's equation
\eqref{232}, to that of the standard rotationally invariant
solution~$f_0$ given by
\begin{equation}
\label{14}
f_0=\frac{1}{1+ \frac{\alpha}{4} r^2}
\end{equation}
already appearing in B.~Riemann's 1854 essay \cite{Ri}.

Given an metric~$\gmetric= f^2 ds^2$ with an arbitrary conformal
factor~$f>0$, we seek estimates of the following type.

We will denote by~$D(\rho)$ the disk of radius~$\rho > 0$ for the
background flat metric of unit area.

\begin{question}
\label{big}
Suppose the torus admits a region of positive Gaussian curvature~$K
\geq \alpha$, which is expressed by a conformal factor~$f$ of
normalized~$L^2$ norm on~$D(\rho)$.  We seek a lower bound for the
variance~$\var(f) \geq N(\rho) \ldots$ where~$N$ is an explicit
function of~$\rho$.
\end{question}

Applying inequalities~\eqref{13} and \eqref{51b}, we can then obtain
the following corollary: Let~$\tau$ be the parameter of the underlying
flat metric.  Then~$ \area(\gmetric) - \Im(\tau) \sys(\gmetric)^2 \geq
N(\rho) \ldots$. In other words, we seek an estimate only dependent on
the lower curvature bound and the size of the disk, modulo a
normalisation of the~$L^2$-norm of~$f$.  Can such a bound be obtained
by studying the reciprocal~$\phi$ instead of~$f$, and using the
convexity of the reciprocal function?  Is Riemann's solution
\eqref{14} optimal as far as the variance is concerned?


We will work with the hypothesis~$K(x,y)\geq \alpha >0$ on the
metric~$f^2 ds^2$ in a disk of radius~$2 \rho>0$.  We need a lower
bound for the variance of~$f=f(\zeta)$ with respect to the standard
measure~$d\zeta$, as already discussed above in Lemma~\ref{sss}.

\begin{lemma}
Assume the Gaussian curvature~$K(x,y)$ satisfies a lower bound~$K\geq
\alpha>0$.  Then the substitution~$\zeta=e^t$ results in a concave
decreasing function~$u(\zeta(t))$ in the variable~$t$.  The second
derivative~$\frac{d^2 u(\zeta(t))}{dt^2}$ is bounded away from zero on
the interval~$\zeta\in [\rho, 2\rho]$ as follows:
$
-\frac{d^2 u}{dt^2} \geq \frac{\alpha \rho}{4} f^2.
$
\end{lemma}

Indeed, since the function~$u=\log \fla$ is rotationally invariant, we
can write~$-\TTT u \geq \alpha \fla^2$, or
\begin{equation}
\label{111}
- \frac{4}{\zeta} u''(t) \geq \alpha \fla^2 
\end{equation}
by \eqref{102b}.  Monotonicity has already been checked in
the Lemma above.  Inequality~\eqref{111} takes the form~$-\frac{d^2
u}{dt^2} \geq \frac{\zeta}{4} \alpha f^2$.

Can the concavity be used to obtain an estimate for the~$t$-variance?
Can one relate the~$t$-variance and the~$\zeta$-variance?

\end{document}